\documentclass[twoside,letterpaper,11pt,leqno]{amsart}

\usepackage{url}
\usepackage[plainpages=false,pdfpagelabels,pdfencoding=auto,psdextra,colorlinks=true,urlcolor=black,linkcolor=black,citecolor=blue,hidelinks]{hyperref}

\usepackage[tmargin=1in, bmargin=1in, lmargin=1.49in, rmargin=1.49in]{geometry}

\usepackage{amsopn}
\usepackage{amsmath}
\usepackage{amsfonts}
\usepackage{amssymb}
\usepackage{graphicx}
\usepackage{epstopdf}
\usepackage{float}
\usepackage{mathrsfs}
\usepackage{mathtools}
\usepackage[inline]{enumitem}
\usepackage{bm}
\usepackage[only,llbracket,rrbracket]{stmaryrd}
\usepackage[caption=false]{subfig}
\usepackage{underscore}
\usepackage{multirow}
\usepackage[capitalize,nameinlink]{cleveref}

\newcommand\where{; \allowbreak \nonscript\; \mathopen{}}

\DeclarePairedDelimiterX\set[1]{\lbrace}{\rbrace}{\nonscript\,  #1 \nonscript\,}

\newcommand\restr[2]{{\left.\kern-\nulldelimiterspace #1 \right\rvert_{#2}}}
\newcommand\restrr[2]{{\kern-\nulldelimiterspace #1 \rvert_{#2}}}

\newcommand{\ol}[1]{\overline{#1}}

\newcommand{\lp}{\left(}
\newcommand{\rp}{\right)}
\newcommand{\lb}{\left[}
\newcommand{\rb}{\right]}
\newcommand{\ldb}{\llbracket}
\newcommand{\rdb}{\rrbracket}

\newcommand{\lv}{\lvert}
\newcommand{\rv}{\rvert}

\newcommand{\lV}{\lVert}
\newcommand{\rV}{\rVert}
\newcommand{\llV}{\left\lV}
\newcommand{\rrV}{\right\rV}

\newcommand\bff{\bm{f}}
\newcommand\bg{\bm{g}}

\newcommand\bp{\bm{p}}
\newcommand\bq{\bm{q}}
\newcommand\br{\bm{r}}

\newcommand\bu{\bm{u}}
\newcommand\bv{\bm{v}}

\newcommand\bx{\bm{x}}

\newcommand\bzero{\bm{0}}

\newcommand\bmu{\bm{\mu}}
\newcommand\brho{\bm{\rho}}

\newcommand\bbR{\mathbb{R}}

\newcommand\cA{\mathcal{A}}
\newcommand\cB{\mathcal{B}}

\newcommand\cE{\mathcal{E}}
\newcommand\cF{\mathcal{F}}

\newcommand\cI{\mathcal{I}}

\newcommand\cO{\mathcal{O}}
\newcommand\cP{\mathcal{P}}

\newcommand\cT{\mathcal{T}}
\newcommand\cU{\mathcal{U}}

\newcommand\tcE{\tilde{\cE}}

\newcommand\ff{\mathfrak{f}}

\newcommand\fA{\mathfrak{A}}

\newcommand\hr{\hat{r}}

\newcommand\hv{\hat{v}}

\newcommand\hB{\hat{B}}

\newcommand\bhr{\bm{\hr}}

\newcommand\bhv{\bm{\hv}}

\newcommand\ta{\tilde{a}}

\newcommand\tu{\tilde{u}}
\newcommand\tv{\tilde{v}}

\newcommand\tB{\tilde{B}}

\newcommand\btu{\bm{\tu}}
\newcommand\btv{\bm{\tv}}

\newcommand\half{\frac{1}{2}}

\newcommand\mone{{-1}}

\newcommand\col{\colon}

\newcommand\dd{\mathop{}\!\mathrm{d}}

\newcommand\grad{\nabla}

\let\div\relax\DeclareMathOperator{\div}{div}

\DeclareMathOperator{\diag}{diag}

\DeclareMathOperator{\Ker}{Ker}
\DeclareMathOperator{\Range}{Range}

\DeclareMathOperator{\myspan}{span}

\DeclareMathOperator{\NNZ}{NNZ}

\allowdisplaybreaks[4]

\newtheorem{theorem}{Theorem}
\newtheorem{lemma}[theorem]{Lemma}
\newtheorem{proposition}[theorem]{Proposition}
\newtheorem{corollary}[theorem]{Corollary}
\theoremstyle{remark}
\newtheorem{remark}[theorem]{Remark}

\numberwithin{theorem}{section}
\numberwithin{equation}{section}

\title[Preconditioning via a Mortar-based Approach]{A Condensed Constrained Nonconforming Mortar-based Approach for Preconditioning Finite Element Discretization Problems}
\author{Delyan Z. Kalchev} 
\address{Center for Applied Scientific Computing, Lawrence Livermore National Laboratory, P.O. Box 808, L-561, Livermore, CA 94551, USA.}
\email{kalchev1@llnl.gov}
\author{Panayot S. Vassilevski} 
\address{Department of Mathematics and Statistics, Portland State University, Portland, OR 97207, USA, and  Center for Applied Scientific Computing, Lawrence Livermore National Laboratory, P.O. Box 808, L-561, Livermore, CA 94551, USA.}
\email{panayot@pdx.edu, vassilevski1@llnl.gov}

\thanks{This work was performed under the auspices of the U.S. Department of Energy by Lawrence Livermore National Laboratory under contract DE-AC52-07NA27344 (LLNL-JRNL-798915).}
\thanks{The work of the second author was partially supported by NSF under grant DMS-1619640.}

\newcommand\Element{\textsc{element}}
\newcommand\Face{\textsc{face}}
\newcommand\Entity{\textsc{entity}}
\newcommand\Elements{\textsc{elements}}
\newcommand\Faces{\textsc{faces}}
\newcommand\Entities{\textsc{entities}}

\newcommand\Edofs{\textsc{e}dofs}

\newcommand\Bdofs{\textsc{b}dofs}

\newcommand\Adofs{\textsc{a}dofs}

\begin{document}

\begin{abstract}
This paper presents and studies an approach for constructing auxiliary space preconditioners for finite element problems using a constrained nonconforming reformulation, that is based on a proposed modified version of the mortar method. The well-known mortar finite element discretization method is modified to admit a local structure, providing an element-by-element or subdomain-by-subdomain assembly property. This is achieved via the introduction of additional trace finite element spaces and degrees of freedom (unknowns) associated with the interfaces between adjacent elements or subdomains. The resulting nonconforming formulation and a  reduced via static condensation Schur complement form on the interfaces are used in the construction of auxiliary space preconditioners for a given conforming finite element discretization problem. The properties of these preconditioners are studied and their performance is illustrated on model second order scalar elliptic problems utilizing high order elements.

\smallskip
\noindent \textsc{Key words.} finite element method, auxiliary space, fictitious space, preconditioning, mortar method, static condensation, algebraic multigrid, element-by-element assembly, high order

\smallskip
\noindent \textsc{Mathematics subject classification.} 65N30, 65N22, 65N55, 65F08
\end{abstract}

\maketitle

\section{Introduction}

The well-known \emph{mortar} finite element discretization method (see, e.g., \cite{DDDiscretizationSolvers,2005Mortar,1999Mortar,1993MortarDD,2001MortarMultipliers}) penalizes jumps across adjacent elements via constraints. This couples degrees of freedom across two neighboring elements and, consequentially, the mortar method does not admit the element-by-element assembly property. In contrast, that property is intrinsic to conforming finite element formulations and it is useful, e.g., in ``matrix-free'' computations since it reduces the coupling across elements. Moreover, the local structure, providing the element-by-element assembly, allows for the utilization of certain element-based coarsening methods, e.g., AMGe (element-based algebraic multigrid) methods \cite{VassilevskiMG} such that an analogous assembly structure is also maintained on coarse levels. Here, based on the simple idea in \cite{IPpreconditioner} of introducing a dedicated space on the element interfaces, the coupling across element boundaries is removed and a convenient local structure, admitting the element-by-element assembly property, is obtained for a modified mortar formulation.

The modification used in this paper is founded upon the following generic idea. The mortar formulation originally employs interface jump constraints and respective Lagrangian multipliers $\mu_\ff$, leading to a jump term $\int_\ff \mu_\ff \ldb u \rdb \dd\rho$ in the resulting Lagrangian functional for each interface $\ff$ between every two adjacent elements $\tau_-$ and $\tau_+$. The proposed modification is to replace this term by two alternative terms, introducing additional interface unknowns $u_\ff$ and other ``one-sided'' jump constraints associated with each $\ff$, leading to Lagrangian multipliers $\mu_{{\tau_-},\ff}$ and $\mu_{{\tau_+},\ff}$. This provides $\int_\ff \mu_{{\tau_-},\ff} (u_- - u_\ff)\dd\rho + \int_\ff \mu_{{\tau_+},\ff} (u_+ - u_\ff)\dd\rho$, where $u_-$ comes from the element $\tau_-$ and $u_+$ from $\tau_+$. Thus, $u_\ff$ represents a trace of the solution on the interface space associated with $\ff$. Clearly, eliminating the unknowns $u_\ff$ recovers the original jump constraints with $\mu_\ff = \mu_{{\tau_-},\ff} = -\mu_{{\tau_+},\ff}$. An important consequence of introducing the additional space of discontinuous (from face to face) functions of the kind $u_b = (u_\ff)$, i.e., piecewise defined functions on the interfaces $\set{\ff}$, is the ability to uniquely relate each of the two new terms with one of the neighboring elements; that is,
$\int_\ff \mu_{{\tau_-},\ff} (u_- - u_\ff)\dd\rho$ with $\tau_-$ and $\int_\ff \mu_{{\tau_+},\ff} (u_+ - u_\ff)\dd\rho$ with $\tau_+$. Accordingly, cross-element coupling occurs only through these new interface unknowns.

The approach exploited in this work to construct preconditioners for a given conforming discretization (an initial formulation with no jump terms involved) is to further replace $\tau_-$ and $\tau_+$ by subdomains $T_-$ and $T_+$ and $\ff$ by the interface $F$ between $T_-$ and $T_+$. The subdomains $\set{T}$ can be viewed as forming a coarse triangulation $\cT^H$ when each $T \in \cT^H$ is a union of elements from an initial fine-scale triangulation $\cT^h$. Any such subdomain $T$ is referred to as an \emph{agglomerate element} or an {\em agglomerate} in short. The resulting modified mortar method employs a pair of discontinuous spaces: one space of functions of the kind $u_e = (u_T)$, i.e., piecewise defined functions on the agglomerates $T \in \cT^H$, and a second space of functions of the kind $u_b = (u_F)$, i.e., piecewise defined functions on the interfaces $\set{F}$ between any two neighboring $T_-$ and $T_+$ in $\cT^H$. This, excluding any forcing terms, leads to a resulting Lagrangian functional with local terms $\half a_T(u_T,v_T) + \sum_{F \subset \partial T} \int_F \mu_{T,F} (u_T - u_F)\dd\rho$ associated with each $T$, where $a_T(\cdot,\cdot)$ is the local version on $T$ of the original symmetric positive definite (SPD) bilinear form $a(\cdot,\cdot)$ from the given conforming discretization. The bilinear form and the respective linear algebra equations of the modified mortar method, possessing the desired local structure, are obtained in a standard way from the problem of finding a saddle point of the Lagrangian functional. This bilinear form and a reduced Schur complement variant of it are utilized in the construction of preconditioners for the original conforming bilinear form $a(\cdot,\cdot)$. Importantly, as shown in this work, while the constrained mortar-based reformulation leads to an indefinite ``saddle-point problem'', the obtained preconditioners are SPD, leading to more natural analysis of their properties and the application of the conjugate gradient method. Moreover, the Schur complement from the reduced via static condensation form is also SPD, allowing the utilization of the abundantly available solvers and preconditioners for systems with SPD matrices like multigrid methods.

The main contribution of the present paper is the introduction and study of a modified mortar reformulation as a technique for obtaining preconditioners for the original conforming bilinear form, utilizing the auxiliary space approach going back to S. Nepomnyaschikh (see \cite{2007DD}) and studied in detail by J. Xu \cite{1996AuxiliarySpace}. Both additive and multiplicative variants of the auxiliary space preconditioners are studied in combination with generic smoothers following the abstract theory in \cite[Theorem 7.18]{VassilevskiMG} by verifying the assumptions stated there. The modified mortar form admits static condensation. Namely, the $u_T$ unknowns and the Lagrangian multipliers can be eliminated using that they are decoupled from each other across elements or subdomains in the modified formulation, obtaining a reduced problem only for the interface unknowns $u_F$. A further advantage of the element-by-element or subdomain-by-subdomain assembly property of the condensed modified mortar formulation is the applicability of the spectral AMGe approach, cf. \cite{2012SAAMGe}, for building algebraic multigrid (AMG) preconditioners for the reduced mortar bilinear form on the interface space which can be viewed as a Schur complement of the full modified mortar form. These auxiliary space preconditioners are implemented and their theoretically shown mesh-independent performance is demonstrated on a scalar second order elliptic problem, including examples with high order elements.

The rest of the paper is organised as follows. \Cref{sec:basic} outlines basic concepts, notation, finite element spaces, and a model problem of interest. The modified mortar approach is introduced in \cref{sec:mortar}, the resulting auxiliary space preconditioners are described in \cref{sec:auxprec}, and \cref{sec:analysisexact} is devoted to the analysis of these preconditioners, showing (\cref{thm:fictopt} and \cref{cor:spectequiv}) the general optimality of a fine-scale auxiliary space preconditioning approach utilizing the mortar reformulation. \Cref{sec:condensation} presents the reduced form and demonstrates (\cref{thm:scopt}) its ability to provide an optimal preconditioning strategy. Numerical results are shown in \cref{sec:numerical}. In the end, \cref{sec:conclusion} provides conclusions and possible future work.

\section{Basics}
\label{sec:basic}

This section is devoted to providing foundations. Notation and abbreviations are introduced to simplify the presentation in the rest of the paper.

\subsection{Mesh and agglomeration}
\label{ssec:meshagg}

A domain $\Omega \subset \bbR^d$ (of dimension $d$) with a Lipschitz boundary, a fine-scale triangulation $\cT^h = \set{\tau}$ of $\Omega$, and a finite element space $\cU^h$ on $\cT^h$ are given. The mesh $\cT^h$ provides a set of elements and respective associated faces, where a \emph{face} is the interface of dimension $d-1$ between two adjacent elements. The focus of this paper is on $\cU^h$ consisting of continuous piecewise polynomial functions equipped with the usual nodal dofs (degrees of freedom). In the rest of the paper, ``$h$'' is used to designate fine-scale entities, whereas ``$H$'' indicates coarse-scale ones.

Let $\cT^H = \set{T}$ be a partitioning of $\cT^h$ into non-overlapping connected sets of fine-scale elements called \emph{agglomerate elements} or simply \emph{agglomerates}; see \cref{fig:aggs}. In general, $\cT^H$ can be obtained by partitioning the \emph{dual graph} of $\cT^h$ --- a graph whose nodes are the elements in $\cT^h$ and any two nodes are connected by an edge in the graph when the respective mesh elements share a face. That is, all $T \in \cT^H$ are described in terms of the elements $\tau \in \cT^h$. In the rest of the paper, capitalization indicates agglomerate entities in $\cT^H$ like {\Element} (short for ``agglomerate element''), \Face, or \Entity, whereas regular letters indicate fine-scale entities in $\cT^h$ like element, face, or entity.

Viewing the {\Elements} in $\cT^H$ as collections of respective fine-scale faces, an intersection procedure over these collections constructs the agglomerate {\Faces} in $\cT^H$ as sets of faces; cf. \cref{fig:faces}. Consequently, each {\Face} can be consistently recognized as the $(d-1)$-dimensional surface that serves as an interface between two adjacent {\Elements} in $\cT^H$. The set of obtained {\Faces} in $\cT^H$ is denoted by $\Phi^H = \set{F}$. Also, the respective sets of $\cU^h$ dofs that can be associated with elements, faces, {\Elements}, and {\Faces} are available. 

For additional information on agglomeration and the topology of ``coarse meshes'' like $\cT^H$, see \cite{VassilevskiMG,2002AMGeTopology}.

\begin{figure}
\centering
\subfloat[][{{\Elements} in 3D}]{\includegraphics[width=0.43\textwidth]{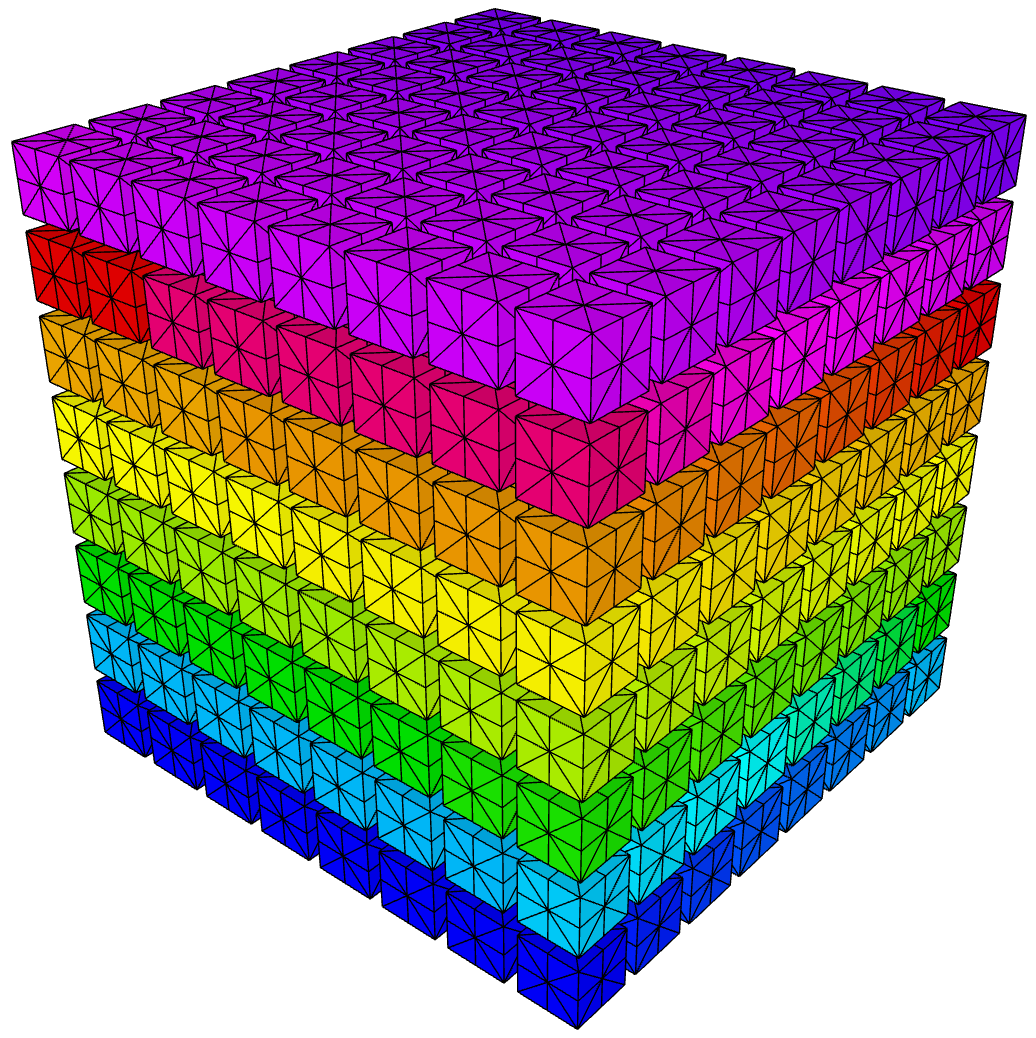}\label{fig:ELEMENTS}}\quad
\subfloat[][{{\Elements} in 2D}]{\includegraphics[width=0.35\textwidth]{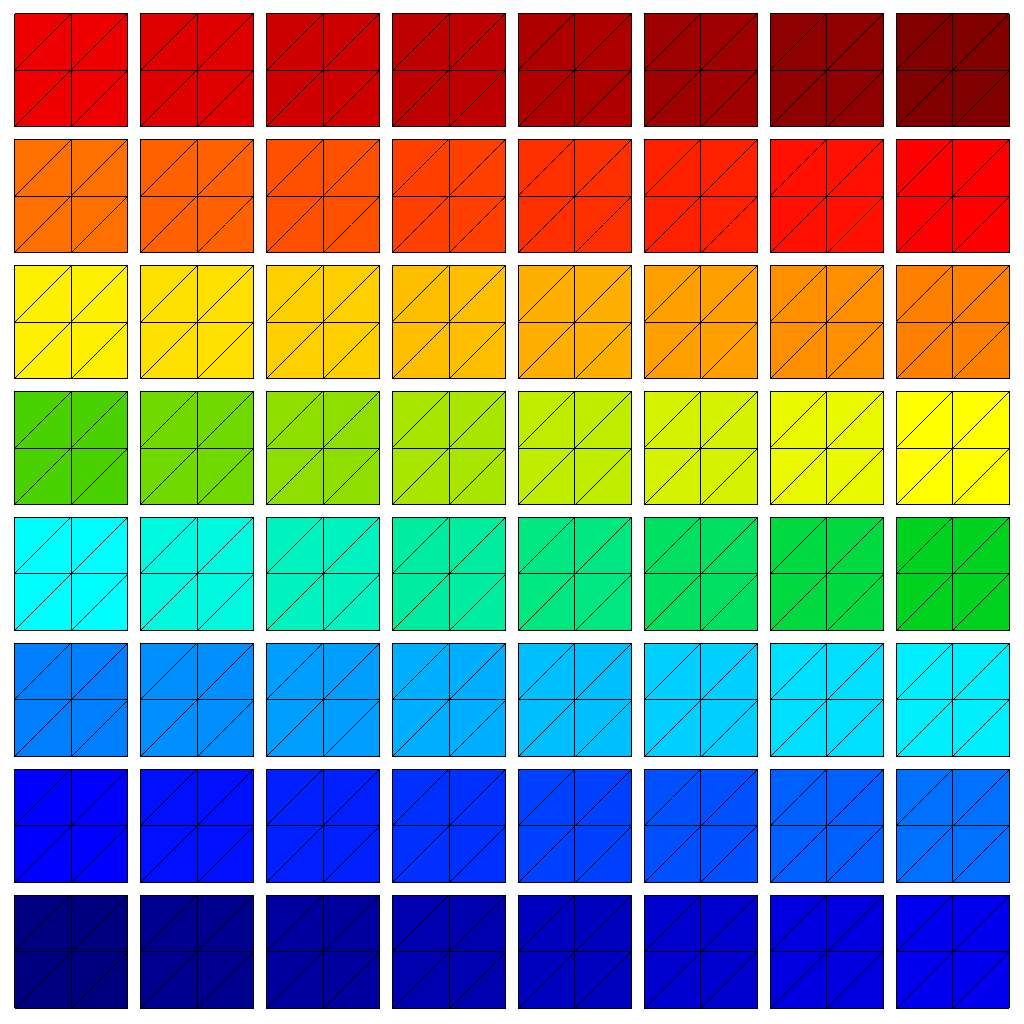}\label{fig:ELEMENTS2D}}
\caption[]{Examples of agglomerates (designated as {\Elements}) of fine-scale elements, utilized in the modified mortar reformulation.}\label{fig:aggs}
\end{figure}

\subsection{Nonconforming spaces}
\label{ssec:spaces}

\begin{figure}
\centering
\includegraphics[width=0.50\textwidth]{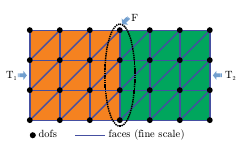}
\caption[]{An illustration of the designation of a {\Face} as a set of fine-scale faces, serving as an interface between {\Elements}.}\label{fig:faces}
\end{figure}

\begin{figure}
\centering
\includegraphics[width=1.00\textwidth]{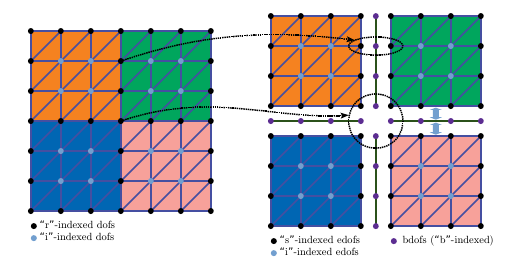}
\caption[]{An illustration of the construction of fine-scale nonconforming finite element spaces $\cE^h$ and $\cF^h$ from a conforming space $\cU^h$, utilizing agglomeration that provides {\Elements} and {\Faces}. Note that the only inter-{\Entity} coupling in the modified mortar reformulation is between {\Elements} and their respective {\Faces}.}\label{fig:nonconfsplitting}
\end{figure}

A main idea in this work is to obtain discontinuous (nonconforming) finite element spaces and formulations on $\cT^H$ and $\Phi^H$. To that purpose, define the finite element spaces $\cE^h$ and $\cF^h$ via restrictions or traces of functions in $\cU^h$ onto $T \in \cT^H$ and $F \in \Phi^H$ respectively. Namely,
\begin{alignat*}{2}
\cE^h &= \Big\lbrace\, v_e^h \in L^\infty(\Omega) \where\;\; &&\forall\, T\in \cT^H,\; \exists\, v^h \in \cU^h\col\;\, \restr{v_e^h}{T} = \restr{v^h}{T} \,\Big\rbrace,\\
\cF^h &= \Big\lbrace\,v_b^h \in L^\infty(\cup_{F\in\Phi^H}F) \where\;\; &&\forall\, F\in \Phi^H,\; \exists\, v^h \in \cU^h\col\;\, \restr{v_b^h}{F} = \restr{v^h}{F} \,\Big\rbrace.
\end{alignat*}
Note that $\cE^h$ and $\cF^h$ are fine-scale spaces despite the utilization of agglomerate mesh structures like $\cT^H$ and $\Phi^H$. Accordingly, the bases in $\cE^h$ and $\cF^h$ are derived via respective restrictions or traces of the basis in $\cU^h$. The degrees of freedom in $\cE^h$ and $\cF^h$ are obtained in a corresponding manner from the dofs in $\cU^h$ as illustrated in \cref{fig:nonconfsplitting}. For simplicity, ``dofs'' is reserved for the degrees of freedom in $\cU^h$, whereas ``edofs'' and ``bdofs'' are reserved for $\cE^h$ and $\cF^h$ respectively. Moreover, ``adofs'' designates the edofs and bdofs collectively and is associated with the product space $\cE^h\times\cF^h$. In more detail, edofs and bdofs are obtained by ``cloning'' all respective dofs for every agglomerate {\Entity} that contains the dofs. Hence, each {\Entity} receives and it is the sole owner of a copy of all dofs it contains and there is no intersection between {\Entities} in terms of edofs and bdofs, i.e., they are completely separated without any sharing, making $\cE^h$ and $\cF^h$ spaces of discontinuous functions. Nevertheless, dofs, edofs, and bdofs are related via their common ``ancestry'' founded on the above ``cloning'' procedure. Thus, the restrictions or traces of vectors or finite element functions in one of the spaces and their representations as vectors or functions in some of the other spaces is seen and performed in a purely ``algebraic'' context. For example, the meaning of $\restr{\bq}{F}$, where $\bq$ is a vector in terms of the edofs of some $T \in \cT^H$ and $F \subset \partial T$, is natural as a vector in terms of the bdofs of $F \in \Phi^H$. This is unambiguous and should lead to no confusion as it only involves a subvector and an appropriate index mapping. In what follows, finite element functions are identified with vectors on the degrees of freedom in the respective spaces.

The portions of vectors corresponding to edofs and bdofs are respectively indexed by ``$e$'' and ``$b$'', leading to the notation $\bhv^T = [\bhv_e^T, \bhv_b^T]^T$ for $\bhv \in \cE^h \times \cF^h$, where $\bhv_e \in \cE^h$ and $\bhv_b \in \cF^h$. By further splitting $\bhv_e^T = [\bhv_i^T, \bhv_s^T]^T$, it is obtained $\bhv^T = [\bhv_e^T, \bhv_b^T]^T = [\bhv_i^T, \bhv_s^T, \bhv_b^T]^T$, where ``$i$'' denotes the edofs in the interiors of all $T \in \cT^H$ and ``$s$'' are the edofs that can be mapped to some bdofs based on the previously described procedure of ``cloning'' dofs into edofs and bdofs. Analogously for $\bv \in \cU^h$, the splitting $\bv^T = [\bv_i^T, \bv_r^T]^T$ is introduced in terms of indexed ``$i$'' dofs in the interiors of all $T \in \cT^H$ and indexed ``$r$'' dofs related to bdofs; see \cref{fig:nonconfsplitting}. Note that there is a clear difference between ``$r$'' and ``$s$'' but also a clear relation. Locally on $T \in \cT^H$, ``$r$'' and ``$s$'' can in fact be equated, but it is necessary to distinguish between ``$r$'' and ``$s$'' indices in a global setting. This should not cause any ambiguity below.

Coarse subspaces $\cE^H \subset \cE^h$ and $\cF^H \subset \cF^h$ can be constructed by respectively selecting linearly independent vectors $\set{\bq_{T,i}}_{i=1}^{m_T}$ for every $T \in \cT^H$ and $\set{\bq_{F,i}}_{i=1}^{m_F}$ for every $F \in \Phi^H$, forming the bases for the coarse spaces. This is an ``algebraic'' procedure formulating the coarse basis functions as linear combinations of fine-level basis functions, i.e., as vectors in terms of the fine-level degrees of freedom. The basis vectors are organized  appropriately as columns of \emph{prolongation} (or \emph{interpolation}) matrices $\cP_e\col \cE^H \mapsto \cE^h$ and $\cP_b\col \cF^H \mapsto \cF^h$, forming $\cP\col \cE^H \times \cF^H \mapsto \cE^h \times \cF^h$ as $\cP = \diag(\cP_e, \cP_b)$. For consistency, the corresponding degrees of freedom associated with the respective coarse basis vectors in $\cE^H$ and $\cF^H$ are respectively called ``\Edofs'' and ``\Bdofs'', and the collective term ``\Adofs'' is associated with $\cE^H\times\cF^H$. In essence, this is based on the ideas in AMGe methods \cite{VassilevskiMG,2002AMGeTopology,2003AMGe,2007AMGe,2008AMGe,2012SAAMGe}. 

\subsection{Model problem}

The model problem considered in this paper is the second order scalar elliptic partial differential equation (PDE)
\begin{equation}\label{eq:pde}
-\div(\kappa\grad u) = f(\bx) \text{ in } \Omega,
\end{equation}
where $\kappa \in L^\infty(\Omega)$, $\kappa > 0$, is a given permeability field, $f \in L^2(\Omega)$ is a given source, and $u \in H^1(\Omega)$ is the unknown function. For simplicity of exposition, the boundary condition $u = 0$ on $\partial \Omega$, the boundary of $\Omega$, is considered, i.e., $u \in H^1_0(\Omega)$. The ubiquitous variational formulation,
\begin{equation}\label{eq:H1min}
\min_{v \in H^1_0(\Omega)} \lb a(v,v) - 2 (f, v) \rb,
\end{equation}
of \eqref{eq:pde} is utilized, providing the weak form
\begin{equation}\label{eq:weakform}
\text{Find } u \in H^1_0(\Omega)\col \; a(u,v) = (f, v), \quad\forall v \in H^1_0(\Omega),
\end{equation}
where $(\cdot,\cdot)$ denotes the inner products in both $L^2(\Omega)$ and $[L^2(\Omega)]^d$, and $a(u,v)=(\kappa \grad u, \grad v)$ for $u,v \in H^1_0(\Omega)$. Consider the fine-scale finite element space $\cU^h \subset H^1_0(\Omega)$ defined on $\cT^h$. Using the finite element basis in $\cU^h$, \eqref{eq:weakform} induces the following linear system of algebraic equations:
\begin{equation}\label{eq:H1linsys}
A \bu = \bff,
\end{equation}
for the global SPD \emph{stiffness matrix} $A$. Moreover, the local on agglomerates $T \in \cT^H$ symmetric positive semidefinite (SPSD) stiffness matrices $A_T$ are obtainable such that (s.t.) $A = \sum_{T \in \cT^H}A_T$ (the summation involves an implicit local-to-global mapping).

\section{Constrained mortar-based formulation}
\label{sec:mortar}

The preconditioners proposed in this paper are based on the ideas of a \emph{mortar} method \cite{DDDiscretizationSolvers,DDTheoryAlgorithms}. The space pair $\cE^h\times \cF^h$ with its degrees of freedom (adofs) is employed together with the ability to construct subspace pairs $\cE^H\times \cF^H$ by selecting basis functions as vectors expressed in terms of the adofs in $\cE^h\times \cF^h$; see the end of \cref{ssec:spaces}. Having the spaces determined, a mortar-based approach is introduced here, where the jumps across interfaces of {\Elements} are penalized via equality constraints. This provides a generic modified mortar reformulation of the original problem, which is utilized as a preconditioner in an auxiliary space framework. In this section, the mortar-based formulation is presented and discussed. In the sections that follow, it is further applied to construct auxiliary space preconditioners for \eqref{eq:H1linsys}, their properties are addressed, and a block-preconditioning technique based on static condensation for the constrained mortar-type problem is described.

Using the prolongation operators defined in the end of \cref{ssec:spaces} and the local on $T$ versions $A_T$ of the fine-scale matrix $A$ in \eqref{eq:H1linsys}, consider the discrete nonconforming constrained quadratic minimization reformulation of \eqref{eq:H1min}
\small
\begin{equation}\label{eq:mortarmin}
\begin{split}
\min &\sum_{T\in\cT^H} \lb [\restr{(\cP_e\bv)}{T}]^T A_T \,\restr{(\cP_e\bv)}{T} - 2 [\restr{(\cP_e\bv)}{T}]^T \bff_T \rb,\\
\text{subject to }\, &Q_{F}\lb\restr{\lp\restr{(\cP_e\bv)}{T}\rp}{F}\rb - \restr{(\cP_b\brho)}{F}  = \bzero, \quad \forall F\in\Phi^H\text{ and }\forall T \in \cT^H \text{ s.t. } \partial T \supset F,
\end{split}
\end{equation}
\normalsize
for $[\bv,\brho] \in \cE^H\times \cF^H$. Here, $\bff_T$ are the local on $T$ versions of $\bff$ in \eqref{eq:H1linsys} and $Q_{F}$ are the $D_{F}$-orthogonal projections onto the local on $F$ spaces spanned by the vectors $\set{\bq_{F,i}}_{i=1}^{m_F}$ associated with $F$ and constituting the basis of $\cF^H$, where $D_F$ is the restriction of the diagonal $D$ of the global $A$ in \eqref{eq:H1linsys} onto the bdofs of $F$. For generality and to avoid over-constraining the formulation, the problem is posed directly on a subspace of $\cE^h\times \cF^h$ requiring the explicit use of the prolongators $\cP_e$ and $\cP_b$. Assume that $\cF^H$ contains the local on {\Faces} constants and $\myspan\set{\bq_{F,i}}_{i=1}^{m_F}$ for each $F \in \Phi^H$ is a proper subspace of the trace space on $F$ of the functions in $\cE^H$. Thus, the constraints provide that the jumps vanish only in a subspace on each {\Face}.

Formulation \eqref{eq:mortarmin} induces the respective global and local on $T$ (modified) mortar matrices
\begin{align}\label{eq:cA}
\cA =
\begin{bmatrix}
\cA_{ee} &C^T &\\
C& &-X\\
&-X^T &
\end{bmatrix},\quad
\cA_T =
\begin{bmatrix}
\cA_{T,ee} &C^T_T &\\
C_T& &-X_T\\
&-X^T_T &
\end{bmatrix}.
\end{align}
Here, $\cA_{ee}$ is associated with the bilinear form $\sum_{T\in\cT^H} [\restr{(\cP_e\bv)}{T}]^T A_T \,\restr{(\cP_e\bu)}{T}$, where $\bv,\bu \in \cE^H$ are respectively test and trial vectors. Also, $C$ represents
\[
\sum_{T\in\cT^H}\sum_{F\subset \partial T} [\cP_F\bmu_{T,F}]^T D_{F} \restr{(\restr{(\cP_e\bu)}{T})}{F},
\]
where $\bmu = [\bmu_{T,F}]_{T\in\cT^H,F\subset \partial T} \in [\cF^H]^2$ is a test Lagrangian multiplier vector for the constraint in \eqref{eq:mortarmin}, while $X$ is associated with
\[
\sum_{T\in\cT^H}\sum_{F\subset \partial T} [\cP_F\bmu_{T,F}]^T D_{F}\, \restr{(\cP_b\bp)}{F},
\]
where $\bp \in \cF^H$ is a trial vector for the interface traces. Here, $\cP_F$ denotes the local on $F$ version of $\cP_b$, i.e., it is the matrix with $\set{\bq_{F,i}}_{i=1}^{m_F}$ for columns. The Lagrangian multipliers are associated with the pairs of {\Elements} and corresponding {\Faces} as they enforce equalities involving ``one-sided'' traces. Observe that $\cA_{T,ee}$, $C_T$, and $X_T$ are simply respective sub-matrices of $\cA_{ee}$, $C$, and $X$ since the assembly of $\cA$ from $\cA_T$ involves only copying without any summation. This is due to the fact that all coupling is through the constraints on the {\Faces}, which is represented by the off-diagonal blocks in $\cA$ and no other connections across {\Elements} and {\Faces} exist. Clearly, $\cA$ and $\cA_T$ are generally symmetric indefinite matrices, where $\cA_{ee}$ and $\cA_{T,ee}$ are SPSD, while $X_T$ is square SPD and $X$ has a full column rank.

Furthermore, denote the respective leading $2\times 2$ block sub-matrices of $\cA$ and $\cA_T$ in \eqref{eq:cA} by
\begin{equation}\label{eq:ucA}
\fA =
\begin{bmatrix}
\cA_{ee} &C^T\\
C& \\
\end{bmatrix},\quad
\fA_{T} =
\begin{bmatrix}
\cA_{T,ee} &C^T_T\\
C_T& 
\end{bmatrix},
\end{equation}
where, as noted above, $\fA = \diag(\fA_T)_{T\in\cT^H}$. The following basic but useful result is obtained.

\begin{lemma}\label{lem:invert}
The matrices $\fA_{T}$ and $\fA$ in \eqref{eq:ucA} are invertible and the $(2,2)$ block of $\fA_{T}^\mone$ is symmetric negative semidefinite (SNSD).
\end{lemma}
\begin{proof}
As long as \eqref{eq:mortarmin} is not over-constrained, provided by the spaces utilized here, $C_T^T$ has a full column rank. Thus, any nonzero vector in the null space of $\fA_{T}$ involves a nonzero local $\bu_T$, i.e., $\bu_T = \restr{\bu}{T}$ for some $\bu \in \cE^H$ such that $Q_F [\restr{(\cP_T \bu_T)}{F}] = \bzero$ on all $F\subset \partial T$ and $A_T[\cP_T \bu_T] = \bzero$, where $\cP_T$ denotes the local on $T$ version of $\cP_e$, i.e., it is the matrix with $\set{\bq_{T,i}}_{i=1}^{m_T}$ for columns. Thus, $\fA_{T}$ is singular if and only if $A_T$ has a nonzero null vector in $\Range(\cP_T) = \myspan\set{\bq_{T,i}}_{i=1}^{m_T}$ with vanishing $Q_F$-projections on all {\Faces}, which is not the case here since the null space of $A_T$ is spanned by the constant vector and $\cF^H$ contains the piecewise constants. That is, $\cA_{T,ee}$ and $C_T$ do not share a common nonzero null vector ($\Ker(\cA_{T,ee}) \cap \Ker(C_T) = \set{\bzero}$); cf. \cite{2005SaddleProblems}. Hence, $\fA_{T}$ and $\fA$ are invertible. Finally, owing to \cite[formula (3.8)]{2005SaddleProblems}, it holds that the $(2,2)$ block of $\fA_{T}^\mone$ is SNSD. Particularly, if $\Ker(\cA_{T,ee}) = \set{\bzero}$ (which is always the case when $\Ker(A_T)= \set{\bzero}$), the $(2,2)$ block of $\fA_{T}^\mone$ is SND (symmetric negative definite).
\end{proof}

\begin{remark}
In view of the full rank of $X_T$, $\cA_T$ in \eqref{eq:cA} is invertible if and only if $\Ker(\cA_{T,ee}) = \set{\bzero}$.
\end{remark}

\begin{remark}
For simplicity, the argument in \cref{lem:invert} makes use of the properties of the particular model problem \eqref{eq:pde}. Namely, it utilizes that the possible null space of $A_T$ is spanned by a constant which cannot vanish on any portion of the boundary of the {\Element}. In general, the result in \cref{lem:invert} holds whenever $A_T$ cannot possess nonzero null-space vectors that vanish on the respective {\Faces}. This is always the case when considering problems coming from PDEs for which Dirichlet-type boundary conditions on portions of the boundary lead to nonsingular problems.
\end{remark}

\Cref{lem:invert} allows the introduction of the following local and global Schur complements resulting from the elimination of all respective {\Edofs} and Lagrangian multipliers from $\cA_T$ and $\cA$ in \eqref{eq:cA}:
\begin{equation}\label{eq:Schur}
\Sigma_T = -[O, -X^T_T]\, \fA_T^\mone\, [O, -X^T_T]^T, \quad \Sigma = -[O, -X^T]\, \fA^\mone\, [O, -X^T]^T,
\end{equation}
where the notation in \eqref{eq:ucA} is used. Observe that $\Sigma_T$ is obtainable by performing only local on $T$ computations and $\Sigma$ can be assembled {\Element} by {\Element} from $\Sigma_T$. This is an important property (employed in \cref{sec:condensation}) resulting from the utilization of interface spaces like $\cF^h$ and $\cF^H$ (\cref{ssec:spaces}), and the particular formulation \eqref{eq:mortarmin}. Now, it is not difficult to establish the following corollary.

\begin{corollary}\label{cor:invert}
It holds that $\Sigma_T$ is SPSD, $\Sigma$ is SPD, and $\cA$ in \eqref{eq:cA} is invertible.
\end{corollary}
\begin{proof}
The SPSD property of $\Sigma_T$ follows immediately from \eqref{eq:Schur} and the SNSD property of the $(2,2)$ block of $\fA_{T}^\mone$ in \cref{lem:invert}. Counting on the presence of essential boundary conditions for the PDE \eqref{eq:pde}, at least one $A_T$ (cf. \eqref{eq:H1linsys}) is nonsingular. Hence, in view of the full rank of $X_T$, at least one $\Sigma_T$ is SPD and the assembly property provides that $\Sigma$ is SPD. Finally, the invertibility of $\fA$ and $\Sigma$ implies the invertibility of $\cA$.
\end{proof}

\section{Auxiliary space preconditioners}
\label{sec:auxprec}

The application of the modified mortar formulation for building auxiliary space preconditioners is addressed now.

Let $\Pi_{h,e}\col \cE^h \mapsto \cU^h$ be a linear transfer operator defined in detail below. Identifying finite element functions with vectors allows to view $\Pi_{h,e}$ as a matrix and obtain $\Pi_{h,e}^T\col \cU^h \mapsto \cE^h$. In order to define the action of $\Pi_{h,e}$, recall that the relation between dofs on one side and edofs on the other is known and unambiguous. Therefore, it is reasonable to define the action of $\Pi_{h,e}$ as taking the arithmetic average, formulated in terms of edofs that correspond to a particular dof, of the entries of a given vector in $\cE^h$ and obtaining the respective entries of a mapped vector defined on dofs. Namely, for any dof $l$ let $J_l$ be the set of corresponding edofs (respective ``cloned'' degrees of freedom) and consider a vector $\bhv_e$ defined in terms of edofs. Then,
\[
(\Pi_{h,e} \bhv_e)_l = \frac{1}{\lv J_l \rv} \sum_{j\in J_l} (\bhv_e)_j.
\]
Clearly, all row sums of $\Pi_{h,e}$ equal 1 and each column of $\Pi_{h,e}$ has exactly one nonzero entry. Assuming that $\cT^h$ is a regular (non-degenerate) mesh \cite{BrennerFEM}, $\lv J_l \rv$ is bounded independently of $h$. That is,
\begin{equation}\label{eq:meshreg}
1 \le \lv J_l \rv \le \varkappa,
\end{equation}
for a constant $1 \le \varkappa < \infty$ which depends only on the regularity of $\cT^h$ but not on $h$. Indeed, let $\varkappa$ be the global maximum number of {\Elements} in $\cT^H$ that a dof can belong to, which in turn is bounded by the global maximum number of elements in $\cT^h$ that a dof can belong to.

Define the transfer operator $\Pi_H\col \cE^H \times [\cF^H]^2 \times \cF^H \mapsto \cU^h$ as 
\begin{equation}\label{eq:PiH}
\Pi_H = [\Pi_{H,e},\, O,\, O],
\end{equation}
where the linear mapping $\Pi_{H,e}\col \cE^H \mapsto \cU^h$ takes the form $\Pi_{H,e} = \Pi_{h,e} \cP_{e}$.

Let $M$ be a ``smoother'' for $A$ in \eqref{eq:H1linsys} such that $M^T + M - A$ is SPD, and $\cB$ is a symmetric (generally indefinite) preconditioner for $\cA$. Define the \emph{additive auxiliary space preconditioner} for $A$
\begin{equation}\label{eq:Badd}
B^\mone_{\mathrm{add}} = \ol{M}^\mone + \Pi_H \cB^\mone \Pi_H^T,
\end{equation}
and the \emph{multiplicative auxiliary space preconditioner} for $A$
\begin{equation}\label{eq:Bmult}
B^\mone_{\mathrm{mult}} = \ol{M}^\mone + (I - M^{-T}A)\Pi_H \cB^\mone \Pi_H^T (I - A M^\mone),
\end{equation}
where $\ol{M} = M (M + M^T -A)^\mone M^T$ is the symmetrized (in fact, SPD) version of $M$. In case $M$ is symmetric, $\ol{M}^\mone$ in $B^\mone_{\mathrm{add}}$ can be replaced by $M^\mone$. The action of $B^\mone_{\mathrm{mult}}$ is obtained via a standard ``two-level'' procedure:

Given $\bv_0 \in \bbR^{\dim(\cU^h)}$, $\bv_\mathrm{m} = B^\mone_{\mathrm{mult}} \bv_0$ is computed by the following steps:
\begin{enumerate}[label=(\roman*)]
\item ``pre-smoothing'': $\bv_1 = M^{-1} \bv_0$;
\item residual transfer to auxiliary space: $\bhr = \Pi_H^T(\bv_0 - A \bv_1)$;
\item auxiliary space correction: $\bhv = \cB^\mone \bhr$;
\item correction transfer from auxiliary space: $\bv_2 = \bv_1 + \Pi_H \bhv$;
\item ``post-smoothing'': $\bv_\mathrm{m} = \bv_2 + M^{-T}(\bv_0 - A \bv_2)$.
\end{enumerate}

If $\cB$ is chosen SPD, then $B^\mone_{\mathrm{add}}$ and $B^\mone_{\mathrm{mult}}$ are clearly SPD. In general, even when the (exact) inverse $\cB^\mone = \cA^\mone$ of $\cA$ in \eqref{eq:cA} is used, it is not immediately obvious whether $B^\mone_{\mathrm{add}}$ and $B^\mone_{\mathrm{mult}}$ are SPD. This is discussed in \cref{sec:analysisexact}.

\subsection*{Smoother}

A particular smoother $M$ for $A$, which is a part of the auxiliary space preconditioners employed in this paper, is shortly described now. Particularly, a polynomial smoother based on the Chebyshev polynomial of the first kind is utilized. For a given integer $\nu \ge 1$, consider the polynomial of degree $3\nu + 1$ on $[0,1]$
\[
p_\nu(t) = \lp 1 - T_{2\nu+1}^2(\sqrt{t}) \rp \frac{(-1)^\nu}{2\nu+1} \frac{T_{2\nu+1}(\sqrt{t})}{\sqrt{t}},
\]
satisfying $p_\nu(0) = 1$, where $T_l(t)$ is the Chebyshev polynomial of the first kind on $[-1, 1]$. Then, $M$ is defined as
\begin{equation}\label{eq:M}
M^\mone = [I - p_\nu(b^\mone D^\mone A)]A^\mone.
\end{equation}
Equivalently, $I - M^\mone A = p_\nu(b^\mone D^\mone A)$, where $b = \cO(1)$ is a parameter satisfying $\bv^T A \bv \le b\, \bv^T D \bv$ for all $\bv \in \cU^h$ and $D$ is either the diagonal of $A$ or another appropriate diagonal matrix. Note that $M$ is SPD and the action of such a polynomial smoother is computed via $3\nu+1$ Jacobi-type iterations using the roots of the polynomial, which makes it convenient for parallel computations; see \cite[Section 4.2.2]{MSthesis}. In practice, $D$ in \eqref{eq:M} can be replaced by a diagonal \emph{weighted $\ell_1$-smoother} like $W = \diag(w_i)_{i=1}^{\dim(\cU^h)}$, where $w_i = \sum_{j=1}^{\dim(\cU^h)} \lv a_{ij} \rv \sqrt{a_{ii}/a_{jj}}$. Such a choice allows setting $b = 1$. More information on the subject can be found in \cite{VassilevskiMG,2011Smoothers,2012SAAMGe,2012ConvSAAMG,1999TGElasticity,2012xinv}.

\section{Analysis}
\label{sec:analysisexact}

Properties of the preconditioners of the type in \eqref{eq:Badd} and \eqref{eq:Bmult} are studied next, showing their optimality in a fine-scale setting. Consider the auxiliary space preconditioners involving the exact inversion of $\cA$
\begin{equation}\label{eq:tB}
\begin{split}
\tB^\mone_{\mathrm{add}} &= \ol{M}^\mone + \Pi_H \cA^\mone \Pi_H^T,\\
\tB^\mone_{\mathrm{mult}} &= \ol{M}^\mone + (I - M^{-T}A)\Pi_H \cA^\mone \Pi_H^T (I - A M^\mone).
\end{split}
\end{equation}
The seemingly apparent auxiliary space here is $\cE^H \times [\cF^H]^2 \times \cF^H$. However, owing to \eqref{eq:mortarmin} and the structure of $\Pi_H$ in \eqref{eq:PiH}, a subspace of $\cE^H$ (more precisely, a subspace of $\cE^H \times \set{0}^2 \times \set{0}$) can be more accurately viewed as the auxiliary space and the properties of $\tB^\mone_{\mathrm{add}}$ and $\tB^\mone_{\mathrm{mult}}$ depend on the properties of the ``subactions'' of $\cA$ and $\Pi_H$ on that subspace.

Clearly, $\brho \in \cF^H$ can be eliminated from \eqref{eq:mortarmin} by replacing the constraints with $Q_{F}[\restr{(\restrr{(\cP_e\bv)}{T_+})}{F} - \restr{(\restrr{(\cP_e\bv)}{T_-})}{F}] = \bzero$ for all $F\in\Phi^H$, where $T_+,T_- \in \cT^H$ denote the respective {\Elements} adjacent to $F$. These altered conditions pose direct constraints on the jumps across {\Faces} of $\bv \in \cE^H$. The modified  minimization problem is equivalent to \eqref{eq:mortarmin} in the sense that it has an identical set of minimizers in $\cE^H$. The constraints can be implicitly imposed by employing the constrained subspace $\tcE^H \subset \cE^H$ defined as
\[
\tcE^H = \set*{\bv \in \cE^H\where Q_{F}[\restr{(\restrr{(\cP_e\bv)}{T_+})}{F} - \restr{(\restrr{(\cP_e\bv)}{T_-})}{F}] = \bzero, \,\forall F\in\Phi^H}.
\]
Obviously, any solution to \eqref{eq:mortarmin} is in $\tcE^H$. Consequently, the unconstrained quadratic minimization over $\bv \in \tcE^H$
\begin{equation}\label{eq:mortarminunconstr}
\min \sum_{T\in\cT^H} \lb [\restr{(\cP_e\bv)}{T}]^T A_T \,\restr{(\cP_e\bv)}{T} - 2 [\restr{(\cP_e\bv)}{T}]^T \bff_T \rb
\end{equation}
is equivalent, in terms of minimizers, to \eqref{eq:mortarmin}. It is trivial that the bilinear form $\ta\col \tcE^H\times \tcE^H \mapsto \bbR$ associated with \eqref{eq:mortarminunconstr} and defined as
\begin{equation}\label{eq:ta}
\ta(\bu, \bv) = \sum_{T\in\cT^H} [\restr{(\cP_e\bv)}{T}]^T A_T \,\restr{(\cP_e\bu)}{T}, \quad\forall\bu,\bv\in\tcE^H \subset \cE^H,
\end{equation}
is SPSD. Here, for the convenience of maintaining consistent vector notation, the basis of $\cE^H$ and its {\Edofs} are employed to represent functions in the constrained subspace $\tcE^H \subset \cE^H$. The invertibility of $\cA$ (\cref{cor:invert}) in \eqref{eq:cA} implies the existence of a unique minimizer in $\tcE^H$ of \eqref{eq:mortarmin} and equivalently of \eqref{eq:mortarminunconstr}, which in turn implies that $\ta(\cdot,\cdot)$ is SPD on $\tcE^H$.

Notice that the matrix associated with $\ta(\cdot,\cdot)$ with respect to the full basis of $\cE^H$ is $\cA_{ee}$ in \eqref{eq:cA}. That is,
\[
\ta(\bu, \bv) = \bv^T \cA_{ee} \bu, \quad\forall\bu,\bv\in\tcE^H \subset \cE^H.
\]
The discussion above demonstrates that while $\cA_{ee}$ is SPSD on the entire $\cE^H$, it is SPD when restricted to the constrained subspace $\tcE^H$, i.e., $\bv^T \cA_{ee} \bv > 0$ for all $\bv \in \tcE^H \setminus \set{\bzero}$ since the constraints filter out its nonzero null vectors ($\ker(\cA_{ee}) \cap \tcE^H = \set{\bzero}$). Hence, the ``inversion operator'' $\cA_{ee}^\mone\col \cE^H \mapsto \tcE^H$ is well-defined as follows: for any $\bg \in \cE^H$, $\cA_{ee}^\mone \bg \in \tcE^H$ is the unique function (vector) that satisfies
\begin{equation}\label{eq:cAeeinv}
\ta(\cA_{ee}^\mone \bg, \bv) = \bv^T \bg,\quad \forall\bv\in \tcE^H.
\end{equation}
That is, when $\bg \notin \tcE^H$, $\cA_{ee}^\mone$ provides the ``least-squares solution'' associated with a minimization of the type in \eqref{eq:mortarminunconstr}. Then, owing to \eqref{eq:PiH} and the equivalence between \eqref{eq:mortarmin} and \eqref{eq:mortarminunconstr}, it holds that $\Pi_H \cA^\mone \Pi_H^T = \Pi_{H,e} \cA_{ee}^\mone \Pi_{H,e}^T$. Thus, $\Pi_H \cA^\mone \Pi_H^T$ is SPSD and the following proposition is obtained.

\begin{proposition}
The preconditioners $\tB^\mone_{\mathrm{add}}$ and $\tB^\mone_{\mathrm{mult}}$ in \eqref{eq:tB} are SPD.
\end{proposition}

Therefore, the desired ``subactions'' of $\Pi_H$, $\cA$, and $\cA^\mone$ are respectively provided by $\Pi_{H,e}$, $\cA_{ee}$, and $\cA_{ee}^\mone$ on the auxiliary space $\tcE^H$ equipped with the norm induced by $\ta(\cdot,\cdot)$, i.e., the $\cA_{ee}$-norm. In this context, the preconditioners in \eqref{eq:tB} can be expressed as
\begin{equation}\label{eq:tB1}
\begin{split}
\tB^\mone_{\mathrm{add}} &= \ol{M}^\mone + \Pi_{H,e} \cA_{ee}^\mone \Pi_{H,e}^T,\\
\tB^\mone_{\mathrm{mult}} &= \ol{M}^\mone + (I - M^{-T}A)\Pi_{H,e} \cA_{ee}^\mone \Pi_{H,e}^T (I - A M^\mone).
\end{split}
\end{equation}

\begin{remark}
Particularly, when $\cE^H = \cE^h$, i.e., no coarsening of $\cE^h$ is employed and $\Pi_{H,e} = \Pi_{h,e}$, then $\Range(\Pi_{H,e}^T) \subset \tcE^H$ and $\Pi_H \cA^\mone \Pi_H^T = \Pi_{H,e} \cA_{ee}^\mone \Pi_{H,e}^T$ is SPD.
\end{remark}

For the rest of this section, the fine-scale case $\cE^H = \cE^h$ is studied, where the respective $\tcE^h$ is consistently defined as
\begin{equation}\label{eq:tcE}
\tcE^h = \set*{\bhv \in \cE^h\where Q_{F}[\restr{(\restrr{\bhv}{T_+})}{F} - \restr{(\restrr{\bhv}{T_-})}{F}] = \bzero, \,\forall F\in\Phi^H},
\end{equation}
i.e., the coarse interface subspace $\cF^H$ is still employed for the jumps. The analysis follows a similar pattern to \cite{IPpreconditioner}. Define the operator $\cI_{h,e}\col  \cU^h \mapsto \tcE^h$ for $\bv \in \cU^h$ via
\[
(\cI_{h,e} \bv)_j = (\bv)_l,
\]
for each edof $j$, where $j \in J_l$ for the corresponding dof $l$. This describes a procedure that appropriately copies the entries of $\bv$ so that the respective finite element functions, corresponding to $\bv$ and $\cI_{h,e} \bv$, can be viewed as coinciding in $H^1(\Omega)$. That is, $\cI_{h,e}$ is an injection (embedding) of $\cU^h$ into $\tcE^h$. As matrices, $\cI_{h,e}$ has the fill-in pattern of $\Pi_{h,e}^T$ with all nonzero entries replaced by 1.

Clearly, $\Pi_{h,e} \cI_{h,e} = I$, the identity on $\cU^h$, implying that $\Pi_{h,e}$ is surjective, i.e., it has a full row rank. Consequently, for any $\bv \in \cU^h$, $\cI_{h,e} \bv \in \tcE^h$ (exactly) approximates $\bv$ in the sense
\begin{equation}\label{eq:approx}
\bv - \Pi_{h,e} \cI_{h,e} \bv = \bzero,
\end{equation}
and $\cI_{h,e} \bv$ is ``energy'' stable since $(\cI_{h,e} \bv)^T \cA_{ee}\, \cI_{h,e} \bv = \bv^T A \bv$ due to the property that $A$ can be assembled from $A_T$ (cf. \eqref{eq:H1linsys}, \eqref{eq:cA}, and \eqref{eq:ta}), which implies
\begin{equation}\label{eq:stable}
\cI_{h,e}^T \cA_{ee}\, \cI_{h,e} = A.
\end{equation}
This is to be expected since in a sense $\tcE^h$ ``includes'' $\cU^h$.

Showing the continuity of $\Pi_{h,e}\col \tcE^h \mapsto \cU^h$ in terms of the respective ``energy'' norms is more challenging and requires $\cF^H$ to satisfy certain properties. Using the indexing notation in \cref{ssec:spaces}, introduce the following splittings for $T\in\cT^H$:
\begin{equation}\label{eq:splitting}
A_T =
\begin{bmatrix}
A_{T,ii} &A_{T,ir}\\
A_{T,ri} &A_{T,rr}
\end{bmatrix},\;
\Pi_{h,e} =
\begin{bmatrix}
\Pi_{ii} &\\
&\Pi_{rs}
\end{bmatrix},\;
\cI_{h,e} =
\begin{bmatrix}
I\\
&\cI_{sr}
\end{bmatrix},
\end{equation}
where $\Pi_{ii} = I$ under an appropriate ordering of the edofs, $\Pi_{rs}$ has exactly one nonzero entry in each column and at least two nonzero entries in each row (exactly two when the respective dofs are in the interior of a {\Face}), and $\cI_{sr}$ is the map from ``$r$'' to ``$s$'' indices. Note that $\cI_{sr}$ is a matrix with the fill-in pattern of $\Pi_{rs}^T$ where all nonzero entries are replaced by 1. Also, for $F \subset \partial T$, let $A_{T,F}$ denote the local on $F$ version of $A_{T,rr}$ in \eqref{eq:splitting} such that $A_{T,rr}$ can be assembled from $A_{T,F}$ for all $F \subset \partial T$. Assume that $\cF^H$ is such that the following local on $F$ property holds for the projection mapping $Q_F$ in \eqref{eq:mortarmin} and some independent of $h$ and $H$ constant $K > 0$:
\begin{equation}\label{eq:bdrapprox}
(\bv_F - Q_F\, \bv_F)^T D_F \,(\bv_F - Q_F\, \bv_F) \le K\, \bv_F^T\, A_{T,F}\,\bv_F,
\end{equation}
for all $F\in\Phi^H$, all $T \in \cT^H$ such that $\partial T \supset F$, and all vectors $\bv_F$ expressed in terms of the fine-scale bdofs on $F$. 

\begin{remark}\label{rem:approx}
Here, \eqref{eq:bdrapprox} represents an ``approximation property'', in ``energy'', of the coarse trace space $\cF^H$ relative to the fine trace space $\cF^h$. It can be interpreted as a measure of quality of approximating ``smooth'' modes on the interfaces. Similar bounds appear in spectral AMGe methods (cf. \cite{2012SAAMGe}) and are obtained via the solution of local generalized eigenvalue problems. For example, in the context here, one possibility would be to use the eigenvectors that correspond to the  eigenvalues in a lower portion of the spectrum of generalized eigenvalue problems of the type $A_{T,F}\, \bv = \lambda D_F\, \bv$ or similar local eigenvalue problems on small patches of elements forming neighborhoods around the {\Faces}. In this work, for simplicity and demonstration purposes, we utilize standard polynomials for the construction of $\cF^H$. Note that, in general, the constant $K$ in \eqref{eq:bdrapprox} may depend on local (in a neighborhood of $F$) quasi-uniformity of the mesh $\cT^h$ and the coefficient $\kappa$ in \eqref{eq:pde}.
\end{remark}

The continuity of $\Pi_{h,e}$ is demonstrated next.

\begin{lemma}\label{lem:approx}
The operator $E_{h,e}\col \tcE^h \mapsto \tcE^h$ defined as $E_{h,e} = \cI_{h,e}\Pi_{h,e} - I$, where $I$ is the identity on $\tcE^h$, is bounded in the sense
\[
(\cI_{h,e}\Pi_{h,e} \bhv - \bhv)^T \cA_{ee} (\cI_{h,e}\Pi_{h,e} \bhv - \bhv) \le 2 \varkappa^2\Lambda K\;\bhv^T \cA_{ee} \bhv,
\]
for all $\bhv \in \tcE^h$, where $K$ is the constant in \eqref{eq:bdrapprox}, $\varkappa$ is from \eqref{eq:meshreg}, and $\Lambda > 0$ is a constant independent of $h$, $H$, the coefficient $\kappa$ in \eqref{eq:pde}, and the regularity of $\cT^h$.
\end{lemma}
\begin{proof}
Using \eqref{eq:splitting}, let $\bv = \Pi_{h,e} \bhv = \lb \bhv_i^T, (\Pi_{rs} \bhv_s)^T \rb^T$, where $\bhv^T = [\bhv_i^T, \bhv_s^T]^T$. Then,
\begin{align*}
\cI_{h,e} \bv &= \cI_{h,e} \Pi_{h,e} \bhv = \lb \bhv_i^T, (\cI_{sr}\Pi_{rs} \bhv_s)^T \rb^T,\\
\btv &= \cI_{h,e} \bv - \bhv = \lb \bzero^T, (\cI_{sr}\Pi_{rs} \bhv_s - \bhv_s)^T \rb^T,
\end{align*}
and
\begin{equation}\label{eq:bdronlyestimate}
\begin{split}
\btv^T \cA_{ee} \btv &= \sum_{T\in\cT^H} \sum_{F\subset \partial T} [\restr{(\Pi_{rs}\bhv_s)}{F} - \restr{(\restr{\bhv_s}{T})}{F}]^T A_{T,F} [\restr{(\Pi_{rs}\bhv_s)}{F} - \restr{(\restr{\bhv_s}{T})}{F}]\\
&\le \Lambda \sum_{T\in\cT^H} \sum_{F\subset \partial T} [\restr{(\Pi_{rs}\bhv_s)}{F} - \restr{(\restr{\bhv_s}{T})}{F}]^T D_{T,F} [\restr{(\Pi_{rs}\bhv_s)}{F} - \restr{(\restr{\bhv_s}{T})}{F}],
\end{split}
\end{equation}
where $D_{T,F}$ is the diagonal of $A_{T,F}$. It is utilized that locally, as noted in \cref{ssec:spaces}, ``$r$'' and ``$s$'' indices can be identified and that the stiffness matrices can be bounded from above by their diagonals with some constant $\Lambda > 0$.

Consider a dof $l$ from the ``$r$'' dofs and the set $J_l$ of related ``$s$'' edofs. Notice that $J_l$ is represented by the $l$-th column of $\cI_{sr}$. Furthermore, let $d_l$ be the corresponding diagonal entry in $A$ and $m_l$, $M_l$ are respectively the minimum and maximum values of $\bhv$ on the edofs in $J_l$. Clearly,
\[
d_l \sum_{j\in J_l}\lb \frac{1}{\lv J_l \rv} \sum_{k\in J_l}(\bhv_s)_k - (\bhv_s)_j \rb^2 \le \lv J_l \rv d_l \lp M_l - m_l \rp^2.
\]
Now, viewing any {\Face} as a collection of respective ``$s$'' edofs, consider a connectivity structure such that two ``$s$'' edofs are connected if they belong to a common {\Face}; see \cref{fig:dofconnection}. Observe that, in terms of this connectivity structure, the ``$s$'' edofs corresponding to $M_l$ and $m_l$ are connected via a path whose length is bounded by $|J_l|$. Following along this path applying the triangle inequality and \eqref{eq:meshreg}, it holds
\[
\lv J_l \rv d_l \lp M_l - m_l \rp^2 \le \varkappa^2 d_l \sum_{(j_+, j_-)\in N_l} \lb (\bhv_s)_{j_+} - (\bhv_s)_{j_-} \rb^2,
\]
where $N_l$ is the set of all pairs of connected ``$s$'' edofs in $J_l$. Hence,
\[
d_l \sum_{j\in J_l}\lb \frac{1}{\lv J_l \rv} \sum_{k\in J_l}(\bhv_s)_k - (\bhv_s)_j \rb^2 \le \varkappa^2 d_l \sum_{(j_+, j_-)\in N_l} \lb (\bhv_s)_{j_+} - (\bhv_s)_{j_-} \rb^2.
\]
Summing over $l$ in the last inequality, in view of \eqref{eq:bdronlyestimate}, \eqref{eq:tcE}, and \eqref{eq:bdrapprox}, provides
\begin{align*}
\btv^T \cA_{ee} \btv &\le \varkappa^2\Lambda \sum_{F \in \Phi^H} \llV \restr{(\restrr{\bhv}{T_+})}{F} - \restr{(\restrr{\bhv}{T_-})}{F} \rrV^2_{D_F}\\
&= \varkappa^2\Lambda \sum_{F \in \Phi^H} \llV (I - Q_F) [ \restr{(\restrr{\bhv}{T_+})}{F} - \restr{(\restrr{\bhv}{T_-})}{F} ] \rrV^2_{D_F}\\
&\le 2 \varkappa^2\Lambda \sum_{T\in\cT^H} \sum_{F\subset \partial T} \llV (I - Q_F) \restr{(\restrr{\bhv}{T})}{F} \rrV^2_{D_F}\\
&\le 2 \varkappa^2\Lambda K  \sum_{T\in\cT^H} \sum_{F\subset \partial T} (\restr{(\restrr{\bhv}{T})}{F})^T A_{T,F}\, \restr{(\restrr{\bhv}{T})}{F} \le 2 \varkappa^2\Lambda K\;\bhv^T \cA_{ee} \bhv,
\end{align*}
where $\lV \cdot \rV_{D_F}$ is the norm induced by $D_F$.
\end{proof}

\begin{corollary}\label{cor:cont}
The operator $\Pi_{h,e}\col \tcE^h \mapsto \cU^h$ is continuous in the sense
\[
(\Pi_{h,e} \bhv)^T A \, \Pi_{h,e} \bhv \le 2(1 + 2\varkappa^2\Lambda K)\, \bhv^T \cA_{ee} \bhv,
\]
for all $\bhv \in \tcE^h$, where the constants are the same as in \cref{lem:approx}.
\end{corollary}
\begin{proof}
The proof is analogous to \cite[Corollary 3.2]{IPpreconditioner}.
\end{proof}

\begin{figure}
\centering
\includegraphics[width=1.00\textwidth]{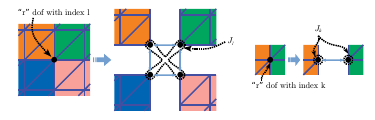}
\caption[]{An illustration of the sets of edofs associated with an ``$r$'' dof and the connectivity between ``$s$'' edofs used in the proof of \cref{lem:approx}.}\label{fig:dofconnection}
\end{figure}

Based on the above properties, the optimality of the auxiliary space preconditioners in a fine-scale setting can be established. Consider first the ``fictitious space preconditioner'' for $A$
\begin{equation}\label{eq:fictprec}
\hB^\mone = \Pi_{h,e} \cB_{ee}^\mone \Pi_{h,e}^T,
\end{equation}
where $\cB_{ee}$ is a symmetric preconditioner for $\cA_{ee}$; cf. \cite{2007DD,1996AuxiliarySpace,VassilevskiMG}.

\begin{theorem}[spectral equivalence]\label{thm:fictopt}
Assume that $\cB_{ee}$ is a spectrally equivalent preconditioner for $\cA_{ee}$ in the sense that there exist positive constants $\alpha$ and $\beta$ such that
\begin{equation}\label{eq:spectequiv}
\alpha^\mone\, \bhv^T \cA_{ee} \bhv \le \bhv^T \cB_{ee} \bhv \le \beta\, \bhv^T \cA_{ee} \bhv,\quad\forall\bhv \in \tcE^h.
\end{equation}
Then, $\hB$ in \eqref{eq:fictprec} is spectrally equivalent to $A$ in \eqref{eq:H1linsys}.
\end{theorem}
\begin{proof}
Similar to \cite[Theorem 3.4]{IPpreconditioner}, \eqref{eq:spectequiv}, \eqref{eq:stable}, and \eqref{eq:approx} imply $\bv^T \hB \bv \le \beta\, \bv^T A \bv$ for all $\bv \in \cU^h$. Conversely, \eqref{eq:spectequiv} and \cref{cor:cont} provide $\bv^T A \bv \le 2\alpha(1 + 2\varkappa^2\Lambda K)\, \bv^T \hB \bv$ for all $\bv \in \cU^h$.
\end{proof}

As in \cite{IPpreconditioner}, it is not difficult to see that the addition of a smoother in \eqref{eq:tB1} does not violate the spectral equivalence in \cref{thm:fictopt}.
\begin{corollary}[spectral equivalence]\label{cor:spectequiv}
Let the smoother $M$ satisfy the property that $M + M^T - A$ is SPD. In the fine-scale setting, i.e., $\Pi_{H,e} = \Pi_{h,e}$, the preconditioners in \eqref{eq:tB1} are spectrally equivalent to $A$. This also holds when $\cA_{ee}^{-1}$ is replaced by $\cB_{ee}^{-1}$ satisfying \eqref{eq:spectequiv} in \eqref{eq:tB1}.
\end{corollary}

There are a couple of additional assumptions on the smoother $M$ in \cite[Theorem 7.18]{VassilevskiMG}. However, they are not necessary in \cref{cor:spectequiv} due to the exactness of the approximation \eqref{eq:approx}. Only the basic property of $A$-convergence (convergence in the norm induced by $A$) of the iteration with $M$ (i.e., $M + M^T - A$ being SPD) is assumed. Nevertheless, when a coarse auxiliary space $\tcE^H$ is utilized, smoothing is necessary.

\subsection*{Convergence and coefficient dependence}

A short formal discussion addressing the dependence of the constant $K > 0$ in \eqref{eq:bdrapprox} on the coefficient $\kappa$ in \eqref{eq:pde} (cf. \cref{rem:approx}) is in order now. Such dependence is determined by the properties of the trace space $\cF^H$. Some intuitive abuse of notation is utilized, which should not lead to confusion. Denote by $\lV \cdot \rV_{a} = \sqrt{a(\cdot,\cdot)}$ the norm induced by the bilinear form $a(\cdot,\cdot)$ in \eqref{eq:weakform}. Consider \eqref{eq:weakform} with an analytic solution $u$ and a finite element approximation $\bu \in \cU^h$ obtained via solving \eqref{eq:H1linsys}, i.e., $\bu = A^\mone \bff$, which is the $a$-orthogonal projection of $u$ onto $\cU^h$. Similarly, a finite element approximation $\btu \in \tcE^h$ such that $\btu = \cA_{ee}^\mone \Pi_{h,e}^T \bff$ is obtained via solving a problem of the type \eqref{eq:mortarminunconstr}. Standard interpolation bounds \cite{BrennerFEM}, the embedding $\cU^h \hookrightarrow \tcE^h$ (via the injection operator $\cI_{h,e}$), the properties of the mortar approach as a nonconforming discretization method \cite{2001MortarMultipliers}, and the continuity of $\Pi_{h,e}$ in \cref{cor:cont} imply an estimate for some independent of $h$ (but possibly dependent on $K$) constant $C > 0$
\[
\lV u - \bu \rV_a \le \lV u - \Pi_{h,e} \btu \rV_a \le C h^{s-1} \lV u \rV_s,
\]
where $\lV \cdot \rV_s$ for a real $s > 1$ is an appropriate Sobolev-type norm. The particular value of $s$ depends on the smoothness of the solution $u$, the finite element spaces, and the properties of the formulations. Results of this kind, although in a bit different setting, are shown in \cite{2001MortarMultipliers} for $s=2$. Thus, $\lV \Pi_{h,e} \cA_{ee}^\mone \Pi_{h,e}^T \bff - A^\mone \bff \rV_a \to 0$ as $h \to 0$ and, accordingly, $\bff^T \Pi_{h,e} \cA_{ee}^\mone \Pi_{h,e}^T \bff\, /\, \bff^T A^\mone \bff \to 1$ as $h \to 0$. This convergence is uniform in $\bff$ whenever $s$ can be bounded away from $1$. This is usually the case under standard regularity lifting properties associated with the solution of problems like \eqref{eq:weakform}. Consequently, as observed in the numerical results in \cref{sec:numerical}, the mortar-based preconditioner improves in quality as $h \to 0$. In short, when the reformulation itself provides a convergent discretization, which is usually the case with mortar-type methods, the obtained preconditioner improves in quality as the discretization is refined. Therefore, the mortar-based approach studied here can be used to derive preconditioning strategies where the coefficient dependence is remedied on sufficiently fine meshes.

\section{Static condensation}
\label{sec:condensation}

The idea here is to build a (block) preconditioner for $\cA$ in \eqref{eq:cA} by eliminating all {\Edofs} and Lagrangian multipliers in $\cA$ involving only local work and preconditioning the resulting Schur complement expressed only on the {\Bdofs}. This procedure is referred to as \emph{static condensation} since it involves ``condensing'' the formulation on the interfaces. Moreover, it maintains the optimality established in \cref{sec:analysisexact}, in a fine-scale setting, of the auxiliary space preconditioners.

Owing to the observations in \cref{lem:invert}, \eqref{eq:Schur}, and \cref{cor:invert}, the {\Edofs} and Lagrangian multipliers can be eliminated from $\cA$ to obtain a ``condensed'' formulation on $\cF^H$ alone. Indeed, using the Schur complements in \eqref{eq:Schur}, consider the block factorization
\[
\cA_T =
\lb
\begin{array}{c|c}
\begin{bmatrix}
\cA_{T,ee} &C^T_T\\
C_T&
\end{bmatrix}
&O\rule[-2.6ex]{0pt}{0pt}\\
\hline\rule{0pt}{2.6ex}
\begin{bmatrix}
O \;\;\,&-X_T^T
\end{bmatrix}
 & \Sigma_T
\end{array}
\rb
\lb
\begin{array}{c|c}
\begin{bmatrix}
I &\\
&I
\end{bmatrix}
&\begin{bmatrix}
\cA_{T,ee} & C^T_T  \\
C_T&
\end{bmatrix}^\mone
\begin{bmatrix}
O\\
-X_T
\end{bmatrix}\rule[-2.6ex]{0pt}{0pt}\\
\hline\rule{0pt}{2.6ex}
O &I
\end{array}
\rb.
\]
Let $S$ be an SPD preconditioner for the global Schur complement $\Sigma$. Since $\Sigma$ is an SPD matrix, which can be assembled from local SPSD matrices $\Sigma_T$ (\cref{cor:invert}), AMGe methods are natural candidates for obtaining $S$. Consequently, the following symmetric, generally indefinite, preconditioner for $\cA$ is obtained:
\begin{equation}\label{eq:cBsc}
\cB_\mathrm{sc}^\mone =
\lb
\begin{array}{c|c}
\begin{bmatrix}
I &\\
&I
\end{bmatrix}
&-\begin{bmatrix}
\cA_{ee} & C^T \\
C&
\end{bmatrix}^\mone
\begin{bmatrix}
O\\
-X
\end{bmatrix}\rule[-2.6ex]{0pt}{0pt}\\
\hline\rule{0pt}{2.6ex}
O &I
\end{array}
\rb
\lb
\begin{array}{c|c}
\begin{bmatrix}
\cA_{ee} &C^T\\
C&
\end{bmatrix}^\mone
&O\rule[-2.6ex]{0pt}{0pt}\\
\hline\rule{0pt}{2.6ex}
-S^\mone\, [O, -X^T] \,\fA^\mone & S^\mone
\end{array}
\rb,
\end{equation}
where $\fA$ is defined in \eqref{eq:ucA}. The preconditioner $\cB_\mathrm{sc}$ is to be utilized within the auxiliary space preconditioners in \eqref{eq:Badd} and \eqref{eq:Bmult}. Applying the action of $\cB_\mathrm{sc}^\mone$ involves invoking $S^\mone$ once and twice $\fA^\mone$, computable via local operations on all $T \in \cT^H$. Owing to \cite[formula (3.8)]{2005SaddleProblems}, it is easy to see that $\Pi_H \cB_\mathrm{sc}^\mone \Pi_H^T$ is SPSD, providing the following result.

\begin{proposition}
The preconditioners $B^\mone_{\mathrm{add}}$ in \eqref{eq:Badd} and $B^\mone_{\mathrm{mult}}$ in \eqref{eq:Bmult}, with $\cB = \cB_\mathrm{sc}$, are SPD.
\end{proposition}

Based on the general analysis of \cref{sec:analysisexact}, it is now demonstrated the optimality, in the fine-scale setting, of the preconditioner choice in \eqref{eq:cBsc} depending on the quality of preconditioning the Schur complement.

\begin{theorem}[spectral equivalence]\label{thm:scopt}
Let the smoother $M$ satisfy the property that $M + M^T - A$ is SPD and consider the fine-scale setting, i.e., $\cE^H = \cE^h$ and $\Pi_H = \Pi_h$. If $S$ is spectrally equivalent to $\Sigma$, then the preconditioners in \eqref{eq:Badd} and \eqref{eq:Bmult}, with $\cB = \cB_\mathrm{sc}$, are spectrally equivalent to $A$.
\end{theorem}
\begin{proof}
In view of the considerations in \cref{sec:analysisexact} leading to \eqref{eq:tB1}, the operator $\cA_{ee}^\mone\col \cE^h \mapsto \tcE^h$, defined via the minimization \eqref{eq:mortarminunconstr} or the corresponding weak form \eqref{eq:cAeeinv}, satisfies $\cA_{ee}^\mone = [I, O, O]\, \cA^\mone \,[I, O, O]^T$, where $I$ is the identity on $\cE^h$. Expressing $\cA^\mone$ similarly to \eqref{eq:cBsc} (by replacing $S$ with $\Sigma$) provides
\[
\cA_{ee}^\mone = [I, O] \lp \fA^\mone + \fA^\mone\, [O, -X^T]^T\, \Sigma^\mone\, [O, -X^T]\, \fA^\mone \rp [I, O]^T.
\]
Define also $\cB_{\mathrm{sc},ee}^\mone\col \cE^h \mapsto \tcE^h$ as $\cB_{\mathrm{sc},ee}^\mone = [I, O, O]\, \cB_\mathrm{sc}^\mone\, [I, O, O]^T$. Hence, due to \eqref{eq:cBsc},
\[
\cB_{\mathrm{sc},ee}^\mone = [I, O] \lp \fA^\mone + \fA^\mone\, [O, -X^T]^T\, S^\mone\, [O, -X^T]\, \fA^\mone \rp [I, O]^T.
\]
Now, let $\gamma\, \bhv_b^T \Sigma^\mone \bhv_b \le \bhv_b^T S^\mone \bhv_b \le \delta\, \bhv_b^T \Sigma^\mone \bhv_b$ for some positive constants $\gamma$ and $\delta$ and all $\bhv_b \in \cF^H$. Then, it is easy to see that $\min\{1,\gamma\}\, \bhv^T \cA_{ee}^\mone \bhv \le \bhv^T \cB_{\mathrm{sc},ee}^\mone \bhv \le \max\{1,\delta\}\, \bhv^T \cA_{ee}^\mone \bhv$ for all $\bhv \in \tcE^h$ (actually, for all  $\bhv \in \cE^h$). Thus, \cref{thm:fictopt} implies that the ``fictitious space preconditioner'' $\Pi_h \cB_\mathrm{sc}^\mone \Pi_h^T = \Pi_{h,e} \cB_\mathrm{sc,ee}^\mone \Pi_{h,e}^T$ is spectrally equivalent to $A$ and the respective result for the preconditioners in \eqref{eq:Badd} and \eqref{eq:Bmult} is due to \cref{cor:spectequiv}.
\end{proof}

\begin{remark}
Notice that the proofs in this paper are algebraic in nature and the particular form of the model problem \eqref{eq:pde}, or \eqref{eq:weakform}, and its properties (particularly, that it is an elliptic PDE) are not utilized. Thus, the mortar reformulation is applicable and its properties are maintained for quite general SPD systems (i.e., convex quadratic minimization problems) that can be associated with appropriate local SPSD versions as long as the interface space $\cF^H$ is selected appropriately to avoid over-constraining the problem (important for \cref{lem:invert} and the sensibility of the formulation) and to provide a trace ``approximation property'' like \eqref{eq:bdrapprox} (important for \cref{lem:approx} and the quality of the auxiliary space preconditioners).
\end{remark}

\section{Numerical examples}
\label{sec:numerical}

\begin{figure}
\subfloat[][Coefficient in 3D]{\includegraphics[width=0.485\textwidth]{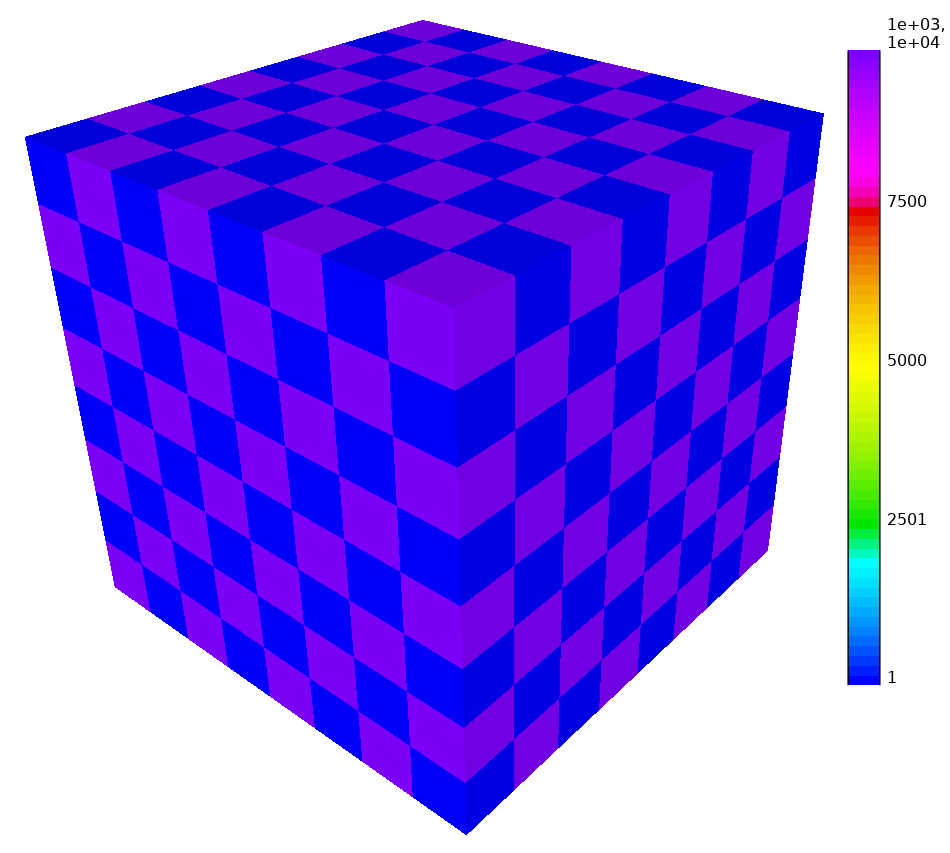}\label{fig:coefficient}}\quad
\subfloat[][Coefficient in 2D]{\includegraphics[width=0.485\textwidth]{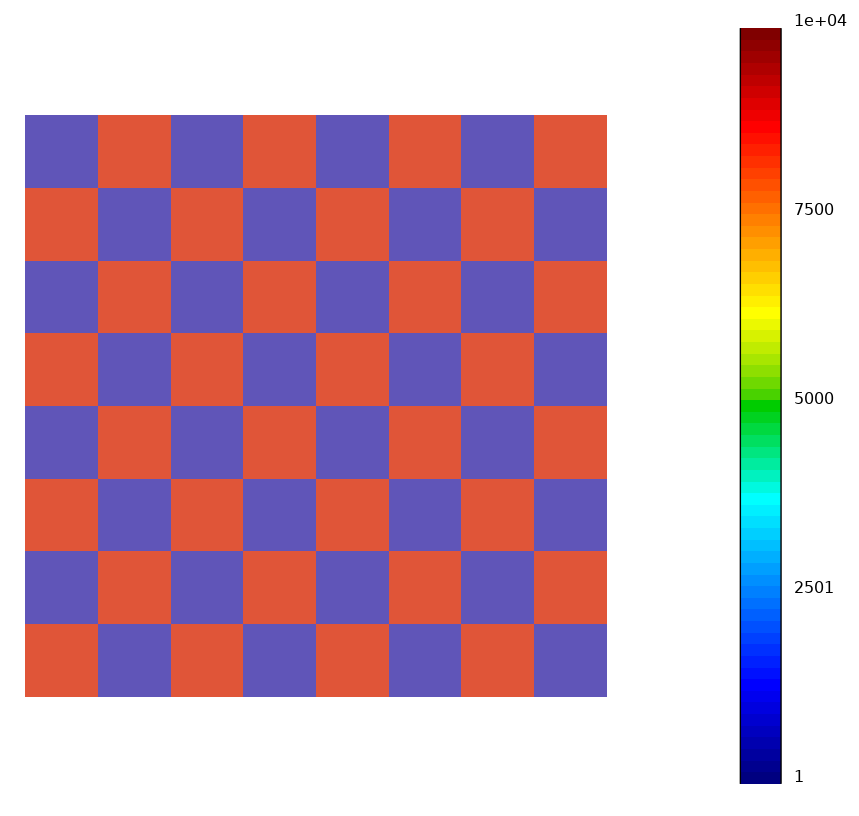}\label{fig:coefficient2D}}
\caption[]{The coefficient $\kappa$ in \eqref{eq:pde}.}
\end{figure}

This section is devoted to numerical results showing two test cases: using low and high polynomial order finite element spaces. The test setting is discussed first.

\subsection{Test setting}

Consider \eqref{eq:pde} with $f \equiv 1$ and the coefficients $\kappa$ with jumps of the kind in \cref{fig:coefficient,fig:coefficient2D}. Here, $\cT^h$ and $\cT^H$ for the mortar reformulation are regular and of the kind shown in \cref{fig:ELEMENTS,fig:ELEMENTS2D}. Note that the {\Elements} in \cref{fig:ELEMENTS,fig:ELEMENTS2D} get accordingly refined as the mesh is refined, i.e., refining $\cT^h$ leads to a respective refinement of $\cT^H$. Nevertheless, the coefficients in \cref{fig:coefficient,fig:coefficient2D} remain unchanged with respect to mesh refinement.

Multigrid methods are invoked for solving or preconditioning the mortar problem in the auxiliary space preconditioners of \cref{sec:auxprec} via employing static condensation, as described in \cref{sec:condensation}, and solving or preconditioning the respective Schur complement problem. A particular multigrid solver employed here is the spectral AMGe method described in \cite[Section 5]{IPpreconditioner}, as implemented in the SAAMGE library \cite{saamge}, taking advantage of the {\Element}-by-{\Element} assembly property of the modified mortar formulation. Any further agglomeration required by the AMGe method is constructed by invoking METIS \cite{metis}. Note that AMGe applied to the mortar form generates a hierarchy of meshes (cf. \cref{ssec:meshagg}), a respective hierarchy of nonconforming spaces (cf. \cref{ssec:spaces}) on the meshes, and a respective hierarchy of mortar formulations (cf. \cref{sec:mortar}), condensed or not, on the spaces.

Having in mind that only the condensed mortar form is presently employed, a few measures of operator complexity (OC), representing relative sparsity in the obtained operator hierarchies, are reported. Namely, the OC of the condensed mortar reformulation:
\[
\mathrm{OC}_\mathrm{m} = 1 + \NNZ(\Sigma)/\NNZ(A),
\]
where $\Sigma$ is the one in \eqref{eq:Schur}; the OC of the ``auxiliary'' multigrid hierarchy relative to the condensed mortar matrix:
\[
\mathrm{OC}_{\mathrm{aux}} = 1 + \sum_{l=1}^{n_\ell} \NNZ(\Sigma^l) / \NNZ(\Sigma);
\]
and the total OC relative to the matrix $A$ in \eqref{eq:H1linsys}:
\[
\mathrm{OC}_{\mathrm{orig}} = 1 + \lp \Sigma + \sum_{l=1}^{n_\ell} \NNZ(\Sigma^l) \rp / \NNZ(A) = 1 + \mathrm{OC}_{\mathrm{aux}} \times (\mathrm{OC}_\mathrm{m} - 1).
\]
Here, NNZ denotes the number of nonzero entries in the sparsity pattern of a matrix and $n_\ell$ is the number of levels (excluding the finest one) in the ``auxiliary'' AMGe hierarchy. The matrices $\Sigma^l$ for $l \ge 1$ represent the coarse versions of $\Sigma$ in the AMGe solver hierarchy for the condensed mortar problem. Recall that dofs are associated with $\cU^h$ and the matrix $A$ in \eqref{eq:H1linsys}, whereas {\Bdofs} are related to $\cF^H$ and the matrix $\Sigma$ in \eqref{eq:Schur}.

In all cases, the preconditioned conjugate gradient (PCG) method and the respective auxiliary space preconditioners in \eqref{eq:Badd} and \eqref{eq:Bmult}, with $\cB^\mone = \cB_\mathrm{sc}^\mone$ in \eqref{eq:cBsc}, are applied for solving the linear system \eqref{eq:H1linsys} and the numbers of iterations $n_\mathrm{a,it}$ and $n_\mathrm{m,it}$ are reported for the respective additive and multiplicative preconditioners. The relative tolerance is $10^{-8}$ and the measure in the stopping criterion is $\br^T B^\mone \br$ for a current residual $\br$, where $B^\mone$ is the utilized preconditioner. The smoother in the end of \cref{sec:auxprec} is employed. Notice that the problem is reduced to the choice of $S^\mone$ as an approximate inverse of the respective $\Sigma$ in \cref{sec:condensation} to completely obtain the action of $\cB_\mathrm{sc}^\mone$.

In all tests, a fine-scale $\cE^H = \cE^h$ is used, while $\cF^H$ is selected as a piecewise polynomial space of a lower order. Note that $\cF^H$ is defined piecewise on the coarse-scale {\Faces}, \emph{not} on the fine-scale faces constituting a {\Face}. That is, e.g., if piecewise constants are used, there is a single constant basis function associated with each {\Face} $F$, \emph{not} multiple basis functions that would correspond to piecewise constants on the faces constituting $F$.

\subsection{Low order discretization}

A 3D mesh (of the type shown in \cref{fig:ELEMENTS}) is sequentially refined and spaces $\cU^h$ of piecewise linear finite elements are constructed. The nonconforming formulation (\cref{sec:mortar}) uses the following spaces: piecewise linear $\cE^H = \cE^h$ in the {\Elements} and piecewise constant $\cF^H$ on the {\Faces}. Static condensation (\cref{sec:condensation}) is relatively cheap in this case, involving the elimination of only a few edofs and Lagrangian multipliers per {\Element}. The smoother in \eqref{eq:M} is used throughout with $\nu = 4$ for the auxiliary space preconditioners. Also, when an AMGe hierarchy (\cite[Section 5]{IPpreconditioner}) is constructed, a single (smallest) eigenvector is taken from all local eigenvalue problems on all levels and \eqref{eq:M} is invoked with $\nu = 2$ as relaxation in the multigrid V-cycle. 

First, $\cB_\mathrm{sc}$ in \eqref{eq:cBsc} is employed with an (almost) exact inversion of the respective Schur complement $\Sigma$ (i.e., $S^\mone = \Sigma^\mone$), resulting in $\cB^\mone = \cB_\mathrm{sc}^\mone = \cA^\mone$ in the auxiliary space preconditioners. Results are shown in \cref{tbl:shurloworderexact}. It is interesting to observe that, as discussed in the end of \cref{sec:analysisexact}, the quality of the preconditioner improves as the mesh is refined, although it takes considerably longer for the additive method to enter such an ``asymptotic regime''. \Cref{tbl:shurlowordersaamge} shows results when $S^\mone$ is implemented via a fixed number of PCG iterations preconditioned by a single V-cycle of AMGe or BoomerAMG \cite{hypre}. Notice that, since the quality of approximation by the mortar method on coarser meshes is lower, initially solving the mortar problem exactly is not beneficial and the approximate inverse actually provides better results, but this is reversed as the mesh is refined and the respective mortar formulation starts producing higher quality approximations.

Notice that the multiplicative method performs substantially better for this hard problem, involving high-contrast coefficients. In view of \eqref{eq:Badd} and \eqref{eq:Bmult}, the smoothing is executed differently in $B^\mone_{\mathrm{add}}$ and $B^\mone_{\mathrm{mult}}$. Here, this results in faster and more robust convergence for the multiplicative method.

\begin{table}
\centering
\small
\captionsetup[subfloat]{width=0.5\textwidth}
\subfloat[][$\cB_\mathrm{sc}^\mone$ with $S^\mone = \Sigma^\mone$]
{\begin{tabular}{ | l | r | r | l | r | r | }
\hline
Refs & \# dofs & \# {\Bdofs} & $\mathrm{OC}_\mathrm{m}$ & $n_\mathrm{a,it}$ & $n_\mathrm{m,it}$ \\
\hline
0 & 4913 & 1344 & 1.400 & 62 & 20 \\
\hline
1 & 35937 & 11520 & 1.487 & 69 & 28 \\
\hline
2 & 274625 & 95232 & 1.536 & 75 & 24 \\
\hline
3 & 2146689 & 774144 & 1.562 & 83 & 22 \\
\hline
4 & 16974593 & 6242304 & 1.576 & 84 & 20 \\
\hline
5 & 135005697 & 50135040 & 1.582 & 86 & 18 \\
\hline
\end{tabular}\label{tbl:shurloworderexact}}
\\
\subfloat[][$\cB_\mathrm{sc}^\mone$ with $S^\mone$ -- four PCG iterations preconditioned by one BoomerAMG V-cycle (additive case) and two PCG iterations preconditioned by one AMGe V-cycle (multiplicative case)]
{\begin{tabular}{ | l | r | l | l | r | r | }
\hline
Refs & $n_\ell+1$ & $\mathrm{OC}_\mathrm{aux}$ & $\mathrm{OC}_{\mathrm{orig}}$ & $n_\mathrm{a,it}$ & $n_\mathrm{m,it}$ \\
\hline
0 & 3 & 1.520 & 1.608 & 64 & 23 \\
\hline
1 & 5 & 1.776 & 1.866 & 70 & 18 \\
\hline
2 & 6 & 1.814 & 1.973 & 77 & 25 \\
\hline
3 & 7 & 1.726 & 1.970 & 85 & 20 \\
\hline
4 & 8 & 1.631 & 1.939 & 92 & 28 \\
\hline
5 & 9 & 1.587 & 1.925 & 99  & 31 \\
\hline
\end{tabular}\label{tbl:shurlowordersaamge}}
\caption[]{Low order test results. The preconditioner $\cB_\mathrm{sc}^\mone$ in \eqref{eq:cBsc} is used, requiring an (approximate) inverse $S^\mone$ of the respective $\Sigma$. Here, either $S^\mone = \Sigma^\mone$ or $S^\mone$ is obtained invoking a fixed number of conjugate gradient iterations preconditioned by a multigrid V-cycle.}
\end{table}

\subsection{High order discretization}

\begin{table}
\centering
\small
\captionsetup[subfloat]{width=0.4\textwidth}
\subfloat[][$\cB_\mathrm{sc}^\mone$ with $S^\mone = \Sigma^\mone$]
{\begin{tabular}{ | l | l | r | r | l | r | r | }
\hline
$\cU^h$ and $\cE^H=\cE^h$ order & $\cF^H$ order & \# dofs & \# {\Bdofs} & $\mathrm{OC}_\mathrm{m}$ & $n_\mathrm{a,it}$ & $n_\mathrm{m,it}$ \\
\hline
2 & 1 & 16641 & 3968 & 1.364 & 39 & 15 \\
\hline
3 & 2 & 37249 & 5952 & 1.193 & 47 & 17 \\
\hline
4 & 3 & 66049 & 7936 & 1.140 & 60 & 21 \\
\hline
5 & 4 & 103041 & 9920 & 1.106 & 62 & 23 \\
\hline
6 & 4 & 148225 & 9920 & 1.058 & 72 & 26 \\
\hline
\end{tabular}}
\\
\captionsetup[subfloat]{width=0.6\textwidth}
\subfloat[][$\cB_\mathrm{sc}^\mone$ with three AMGe V-cycles as $S^\mone$]
{\begin{tabular}{ | l | l | r | l | l | r | r | }
\hline
$\cU^h$ and $\cE^H=\cE^h$ order & $\cF^H$ order & $n_\ell+1$ & $\mathrm{OC}_\mathrm{aux}$ & $\mathrm{OC}_{\mathrm{orig}}$ & $n_\mathrm{a,it}$ & $n_\mathrm{m,it}$ \\
\hline
2 & 1 & 2 & 1.250 & 1.455 & 48 & 17 \\
\hline
3 & 2 & 2 & 1.111 & 1.215 & 60 & 21 \\
\hline
4 & 3 & 4 & 1.710 & 1.240 & 65 & 26 \\
\hline
5 & 4 & 4 & 1.178 & 1.125 & 85 & 31 \\
\hline
6 & 4 & 5 & 1.495 & 1.087 & 90 & 31 \\
\hline
\end{tabular}}
\caption[]{High order test results. The preconditioner $\cB_\mathrm{sc}^\mone$ in \eqref{eq:cBsc} is used, requiring an (approximate) inverse $S^\mone$ of the respective $\Sigma$.}\label{tbl:highorder}
\end{table}

Results in 2D are shown using a fixed mesh of the type in \cref{fig:ELEMENTS2D} and increasing the polynomial order. The smoother in \eqref{eq:M} is used throughout with $\nu = 4$ for the auxiliary space preconditioners, while \eqref{eq:M} is invoked with $\nu = 2$ as relaxation in the multigrid V-cycle of the AMGe method (\cite[Section 5]{IPpreconditioner}). Since no mesh refinement is performed, no additional agglomeration is invoked for the AMGe hierarchy. That is, the {\Elements} for the mortar reformulation are maintained throughout the hierarchy and only the basis functions are reduced during the construction of the AMGe solver hierarchy. This is reminiscent of $p$-multigrid but in a spectral AMGe setting.

Results are shown in \cref{tbl:highorder}. Notice that even when utilizing the exact auxiliary space reformulation (the case of $\cB_\mathrm{sc}^\mone = \cA^\mone$), the number of iterations slightly increases. The quality of the preconditioner is dependent on the choice of the space $\cF^H$. While the method admits considerable flexibility in selecting $\cF^H$, the tests here utilize simple polynomial spaces which, as explained above, are piecewise defined on the coarse-scale {\Faces}, whereas $\cE^h$ is of piecewise polynomials on fine-scale entities. Therefore, as the polynomial order is increased, $\cF^H$ potentially provides relatively worse approximations of the traces of functions in $\cE^h$. This can be remedied by selecting richer spaces for $\cF^H$. Nevertheless, even the simple choices here provide good results. Since the multiplicative method provides better smoothing, it is not a surprise that it performs better and, while maintaining all other parameters the same, it leads to better  robustness with respect to deficiencies in the trace space  $\cF^H$ and with respect to the quality of the  preconditioner for the ``auxiliary'' mortar problem.

\section{Conclusions}
\label{sec:conclusion}
In this paper, we have proposed and studied a modified version of a mortar finite element discretization method by forming an additional finite element space of discontinuous functions on the interfaces between (agglomerate) elements, or subdomains, utilized for coupling local bilinear forms via equality constraints on the interfaces. This modification allows for the agglomerate-by-agglomerate or subdomain-by-subdomain assembly. Then, the resulting modified mortar formulation and a respective reduced version obtained via static condensation on the interfaces are used in combination with polynomial smoothers for the construction of auxiliary space preconditioners. They are analyzed and their proven mesh-independent spectral equivalence is demonstrated in numerical results for 2D and 3D second order scalar elliptic PDEs, including the applicability for high order finite element discretizations. The local structure of the modified condensed mortar bilinear form, providing an agglomerate-by-agglomerate assembly property, is further useful in the obtainment of element-based algebraic multigrid (AMGe) utilized to approximate the inverse of the Schur complement in the auxiliary space preconditioners that involve static condensation. A possible practical extension of this work is the application of the proposed auxiliary space preconditioners in the setting of ``matrix-free'' solvers for high order finite element discretizations in combination with AMGe coarse solves and polynomial smoothers like the one outlined in the end of \cref{sec:auxprec}. Also, considering a variety of different choices for the interface space $\cF^H$ can lead to broader applicability of the mortar reformulation and improve its robustness in general settings. A possible continuation and extension of this work would be to study  different options for constructing $\cF^H$ that satisfy \eqref{eq:bdrapprox} or search for other conditions on $\cF^H$ that can provide the desired spectral equivalence properties demonstrated in this paper.

\bibliographystyle{plainurl}
\bibliography{references}

\end{document}